\documentclass[11pt]{amsart}
\usepackage{amsfonts, amsmath, amssymb, color}
\usepackage{etex}   
\usepackage{amsthm}
\usepackage{booktabs}
\usepackage{hhline}
\usepackage{lipsum}
\usepackage{lscape}
\usepackage{subfig}
\usepackage{textcomp}
\usepackage{tikz}
\usetikzlibrary{decorations.pathmorphing,shapes,arrows,positioning}
\usepackage[all]{xy}
\usepackage{xcolor}
\usepackage{young}
\usepackage{youngtab}
\usepackage{ytableau}
\usepackage{hyperref}
\hypersetup{colorlinks,linkcolor=blue,urlcolor=cyan,citecolor=blue}

\allowdisplaybreaks[4]
\theoremstyle{plain}
\newtheorem{thm}{Theorem}[section]
\newtheorem{lem}[thm]{Lemma}
\newtheorem{prop}[thm]{Proposition}
\newtheorem{cor}[thm]{Corollary}
\newtheorem{rem}[thm]{Remark}

\theoremstyle{definition}

\newcommand{\la}{\lambda}
\newcommand{\brho}{\boldsymbol{\rho}}

\newcommand{\bl}{\boldsymbol{\lambda}}
\newcommand{\bmu}{\boldsymbol{\mu}}
\newcommand{\bnu}{\boldsymbol{\nu}}

\numberwithin{equation}{section} \errorcontextlines=0

\newcommand{\GL}{\mathrm{GL}}

\begin{document}
\title{Characters of $\GL_n(\mathbb F_q)$ and vertex operators}
\author{Naihuan Jing}
\address{School of Mathematics, South China University of Technology, Guangzhou, Guangdong 510640, China}
\address{Department of Mathematics, North Carolina State University, Raleigh, NC 27695, USA}
\email{jing@ncsu.edu}
\author{Yu Wu}
\address{School of Mathematics, South China University of Technology,
Guangzhou, Guangdong 510640, China}
\email{wywymath@163.com}
\subjclass[2010]{Primary: 20C33, 17B69; Secondary: 05E10}\keywords{Finite general linear groups, characters,
vertex operators, Hall-Littlewood functions, Schur functions}
\thanks{The paper is partially supported by
NSFC grant No. 12171303.}

\maketitle
\begin{abstract}
In this paper, we present a vertex operator approach to construct and compute all complex irreducible characters of the general linear group
$\GL_n(\mathbb F_q)$. Green's theory of $\GL_n(\mathbb F_q)$ is recovered and enhanced under the realization of the Grothendieck ring of representations
$R_G=\bigoplus_{n\geq 0}R(\GL_n(\mathbb F_q))$ as two isomorphic Fock spaces associated to two infinite-dimensional $F$-equivariant Heisenberg Lie algebras $\widehat{\mathfrak{h}}_{\hat{\overline{\mathbb F}}_q}$ and $\widehat{\mathfrak{h}}_{\overline{\mathbb F}_q}$, where $F$ is the Frobenius automorphism of
the algebraically closed field
$\overline{\mathbb F}_q$. Under this picture, the irreducible characters are realized by the Bernstein vertex operators for
Schur functions, the characteristic functions of the conjugacy classes are realized by the vertex operators
for the Hall-Littlewood functions,
and the character table is completely given by matrix coefficients of vertex operators of these two types. One of the features of the current approach is a simpler identification of the
Fock space $R_G$ as the Hall algebra of symmetric functions via vertex operator calculus, and another is that we are able to
compute in general the character table, where Green's degree formula is demonstrated as an example.
\end{abstract}

\section{Introduction}

Let $\GL_n(\mathbb F_q)$ be the general linear group of $n$ by $n$ matrices over the finite field $\mathbb F_q$. In a remarkable paper \cite{jaG}  Green
determined all irreducible characters of $\GL_n(\mathbb F_q)$ by parabolic induction, generalizing Frobenius' theory of the symmetric groups.
Green's theory has greatly influenced the developments of representation theory and algebraic combinatorics.
Zelevinsky used a Hopf algebra approach \cite{Ze} to study representations of $\GL_n(\mathbb F_q)$ and Springer and Zelevinsky \cite{SZ} further developed the approach on unipotent characters.
The Green polynomials have been generalized to finite Chevalley groups. Deligne and Lusztig \cite{DL} generalized Green's parabolic induction to construct linear representations of finite groups of Lie types using $\ell$-adic cohomology
(cf.\cite{C}) and later Lusztig found all representations of finite simple groups of Lie types \cite{L, GM}.

In algebraic combinatorics, Green's theory showed the importance of then a new family of symmetric functions--Hall-Littlewood functions.
 Macdonald has written the famous book \cite{Mac} to account Green's identification of the Hall algebra with the ring of symmetric functions (for a survey of Green's theory and cusp forms, see \cite{S}).
 Green's idea has also played an important role in understanding of quantum groups via Hall's algebras \cite{Green2}.

In this paper, we present a new method to construct and compute all complex irreducible characters of $\GL_n(\mathbb F_q)$ using representation theory of infinite dimensional Lie algebras,
which follows the general strategy of using vertex algebras
to reconstruct Frobenius' theory of the symmetric group $\mathfrak S_n$ \cite{Jing1}, Schur's theory of the spin group $\tilde{\mathfrak S}_n$ \cite{Jing1}
and Specht's theory of the wreath product $G\wr \mathfrak S_n$ \cite{FJW} etc.. The corresponding
vertex operators are homogeneous vertex representations, twisted vertex representations of
the affine Lie algebra $\widehat{\mathfrak sl}_2$ and the affine Lie algebra $\widehat{\mathfrak gl}_n$.
In the current work we will introduce $\Phi^*$-dependent vertex operators
and  $\Phi$-dependent vertex operators, where $\Phi$ and $\Phi^*$ are the $F$-orbits of the multiplicative group $M_n=\mathbb F_{q^n}^{\times}$ of the Galois field $\mathbb F_{q^n}$ and the character group $M^*_n=\hat{\mathbb F}_{q^n}^{\times}$, in some sense it is the generalized affine Lie algebras $\widehat{gl}_{\infty}\rtimes \overline{\mathbb F}_q^{\times} $ and
$\widehat{gl}_{\infty}\rtimes \hat{\overline{\mathbb F}}_q^{\times}$, where $\overline{\mathbb F}_q^{\times}$ is the multiplicative group of the algebraic closure $\overline{\mathbb F}_q$ and $\hat{\overline{\mathbb F}}_q^{\times}$ its dual group.

Our work can be illustrated by the following flow chart:

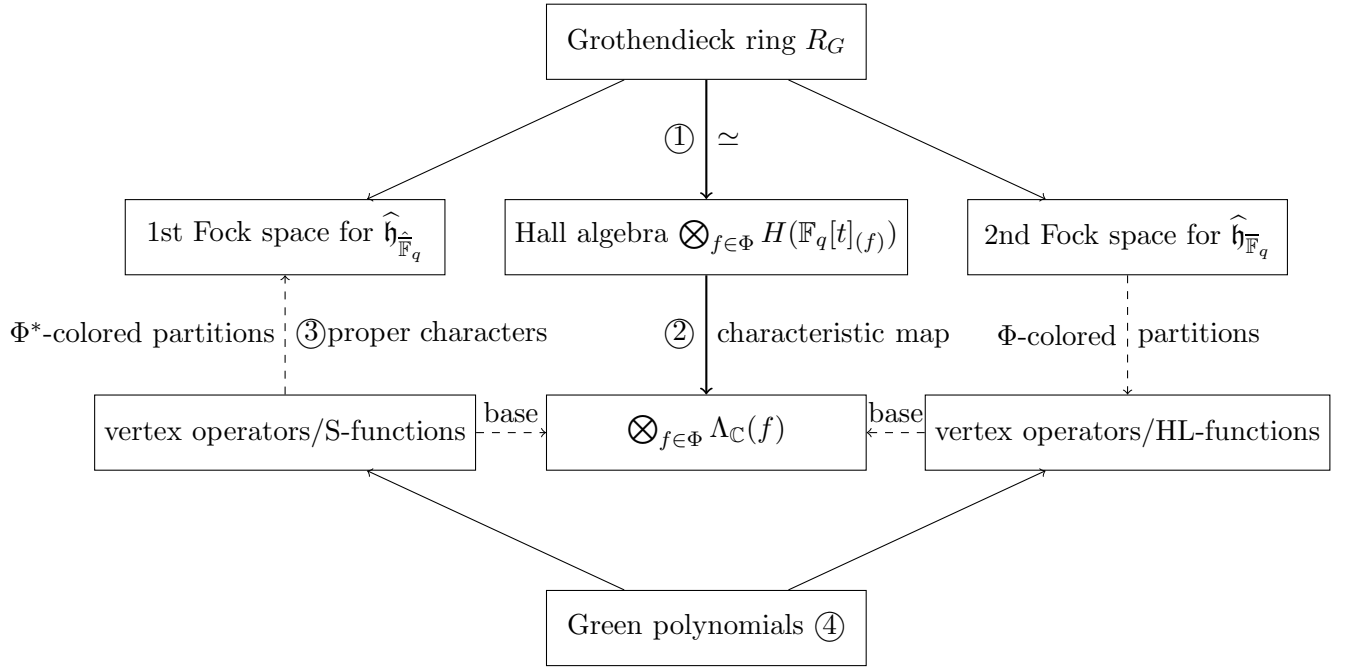
\begin{gather*}
  \centering
\begin{tikzpicture}[node distance = 2.6cm]
 \tikzset{
 base/.style={draw = black,minimum width=0.25\columnwidth,minimum height = 1cm},
 startstop/.style = {base,rectangle},
 process/.style={base,rectangle},
 io/.style={base,rectangle},
 arrow/.style ={->,thick},
 }
 \node(Grot)[startstop]{Grothendieck ring $R_G$};
 \node(HA)[io,below of=Grot]{Hall algebra $\bigotimes_{f\in\Phi}H(\mathbb F_q[t]_{(f)})$};
 \node(Lambda)[process,below of =HA]{$\bigotimes_{f\in\Phi}\Lambda_{\mathbb{C}}(f)$};
\node(Green)[process,below of =Lambda]{Green polynomials $\textcircled{4}$};
\node(conj)[process, right of=HA, xshift=3cm]{2nd Fock space for $\widehat{\mathfrak h}_{\overline{\mathbb F}_q}$};
\node(char)[process, left of =HA,xshift=-3cm]{1st Fock space for $\widehat{\mathfrak h}_{\hat{\overline{\mathbb F}}_q}$};
\node(Schur)[process,below of =char]{vertex operators/S-functions};
\node(HL)[process,below of =conj]{vertex operators/HL-functions};
\draw[arrow](Grot)--node[right]{$\simeq$}node[left]{$\textcircled{1}$}(HA);
\draw[arrow](HA)--node[right]{characteristic map}node[left]{$\textcircled{2}$}(Lambda);
\draw[->](Green)--(HL);
\draw[->](Green)--(Schur);
\draw[->](Grot)--(conj);
\draw[->](Grot)--(char);
\draw[dashed,->](Schur)--node[right]{$\textcircled{3}$proper characters}node[left]{$\Phi^*$-colored partitions}(char);
\draw[dashed,->](conj)--node[left]{$\Phi$-colored }node[right]{partitions}(HL);
\draw[dashed,->](HL)--node[above]{base}(Lambda);
\draw[dashed,->](Schur)--node[above]{base}(Lambda);
\end{tikzpicture}
\end{gather*}

Let us explain our current work in details.
Let $R(G_n)$ be the
Hilbert space of complex class functions of $\GL_n(\mathbb F_q)$ equipped with the
canonical inner product. The conjugacy classes are parameterized by
weighted partition-valued functions on the set $\Phi$ of the $F$-orbits of $\overline{\mathbb F}^{\times}_q$, namely the conjugacy classes are parameterized by $\Phi$-colored weighted partitions ${\bl}=({\bl}(f))_{f\in\Phi}\subset \mathcal P$
such that
\begin{equation}
||{\bl}||=\sum_{f\in\Phi}d(f)|\bl(f)|=n.
\end{equation}
For each such weighted $\Phi$-colored partition
$\bl: \Phi\mapsto \mathcal P$, let $\pi_{\bl}$ be the characteristic
function at the conjugacy class $c_{\bl}$. Then $\{\pi_{\bl}\}$ form an orthogonal basis of $R(G_n)$ under the canonical inner product.

We introduce the infinite-dimensional graded space
\begin{equation}
R_G=\bigoplus_{n=0}^{\infty}R(G_n)
\end{equation}
as the ring
equipped with the multiplication given by $f\cdot g=\mathrm{Ind}_{P_{(m, n)}}^{G_{m+n}} f\otimes g$ for any two homogeneous class functions
$f\in R(G_m), g\in R(G_n)$. Here $P_{(m, n)}$ is a parabolic subgroup of $\GL_n(\mathbb F_q)$.
We first realize $R_G$ as the Fock space $\mathcal{F}_G$ of the infinite dimensional $F$-equivariant Heisenberg algebra $\widehat{\mathfrak{h}}_{\hat{\overline{\mathbb F}}_q}$ generated by
$p_n(\phi)$ $(n\neq 0, \phi\in\Phi^*)$  and the central element $C$ subject to the relations
\begin{align}
[p_m(\phi), p_n(\phi')]=m\delta_{m, -n}\delta_{\phi, \phi'}C
\end{align}
which equips the Fock space with the canonical inner product $\langle p_{\lambda}(\phi), p_{\mu}(\phi')\rangle=z_{\lambda}\delta_{\phi, \phi'}$
consistent with the canonical inner product from $R_G$.
We use the vertex operator realization of Schur functions to construct an orthonormal basis of the Fock space $\mathcal{F}_G$ indexed by
$\Phi^*$-colored weighted partitions $\tilde{\bl}: \Phi^*\longrightarrow\mathcal P$ such that $||\tilde{\bl}||=n$.
For each $\phi\in\Phi^*$, we define the Bernstein vertex operators for Schur functions (cf. \cite{Ze, Jing3})
\begin{equation}
S(\phi; z)=:\exp(\sum_{n\neq 0}\frac{1}{n}p_n(\phi)z^{n}):=\sum_{n\in\mathbb Z} S_n(\phi)z^n,
\end{equation}
where $:\ \ :$ denotes the normal order (see \eqref{e:Schur1}).
By the vertex operator realization of the Schur functions \cite{Jing2}, we have an orthonormal basis $S_{\tilde{\bl}}:=\prod_{f\in\Phi^*}S_{\tilde{\bl}(f)}.1$, which are certain products of the vertex operators $S_n(f)'s$ (cf. \cite{Jing3}).

Under the realization of $R_G$ as the tensor product of the ring of symmetric functions over $\widehat{\mathfrak{h}}_{\hat{\overline{\mathbb F}}_q}$, we can use the discrete Fourier transform to change bases from the
product of the power-sums $p_n(\phi)'s$ to that of $p_n(f)'s $, where $f$ is an $F$-orbit of $\mathbb F_{q^n}^{\times}$.
We show that the Fock space $R_G$ is also identified with the (2nd) Fock space of the infinite-dimensional
  $F$-equivariant Heisenberg algebra
  $\widehat{\mathfrak{h}}_{\overline{\mathbb F}_q}$ generated by
  the $p_n(f)$ ($n\neq 0, f\in\Phi$) and the central element $C$ subject to
 \begin{align}
[p_m(f), p_n(f')]=\frac{m}{q^{d(f)m}-1}\delta_{m, -n}\delta_{f, f'}C, \qquad\quad (m>0)
\end{align}

We then recall the vertex operator for Hall-Littlewood functions introduced in \cite{Jing2}: for each $f\in\Phi^*$
\begin{equation}
Q(f; z)=:\exp(\sum_{n\neq 0}\frac{q_f^n-1}{n}p_n(f)z^{n}):=\sum_{n\in\mathbb Z} Q_n(f)z^n,
\end{equation}
where $: \ \ :$ is the normal order and the exact formula of $Q(f; z)$ will be given in \eqref{e:VO1}.
We use the vertex operator approach to show that $\Lambda_{\mathbb C}$ obey the multiplication of the Hall algebra
and the Fock space $R_G$ is ring isomorphic to the Hall algebra under the
characteristic map $ch$ defined by:
\begin{equation}
ch(\pi_{\bl})=c(\bl)Q_{\bl},
\end{equation}
where $c(\bl)$ is a constant in
$q$ (see \eqref{e:chb}) and $Q_{\bl}$ is certain product of the vertex operators $Q_n(f)'s$. This isomorphism also establishes that $Q_{\lambda}(f)$ corresponds to
the Hall-Littlewood symmetric functions associated with $\lambda$ in the $p_n(f)$, which is an independent proof of this important result using vertex operators. We remark that Macdonald's description
of the Hall algebra connection occupied
three chapters in \cite{Mac} for
this part of Green's theory. We also note that
our original identification of $Q_{\lambda}(f)$ by the first author \cite{Jing2} using the vertex operators had used
Littlewood's raising formula of the Hall-Littlewood function. In this regard, our current approach provides a simpler and quicker proof of this
identification of the Hall algebra based on vertex operators.

According to the two identifications of the representation ring $R_G$ as both the Fock spaces spanned by bases indexed by $\Phi^*$-colored weighted partitions and $\Phi$-colored weighted partitions,
the orthonormal set $\{S_{\tilde{\bl}}\}$ should provide the complete set of irreducible characters provided that we verify that $ch^{-1}(S_{\tilde{\bl}})$ are proper characters. By Brauer's theorem of modular characters \cite{B}, the set $\{S_{\tilde{\bl}}\}$ correspond to virtual characters, thus
$\{\pm S_{\tilde{\bl}}\}$ provide the complete list of irreducible characters. We finally use vertex operators to compute the degree of the virtual character $ch^{-1}(S_{\tilde{\bl}})$
to settle the problem to determine all irreducible characters.
We also obtain a general formula of the irreducible characters which reduces the computation to
characters of the symmetric groups $\mathfrak S_m$ and Green's polynomials.

Recently, Liu and one of us \cite{JL} have used the vertex operator to derive an iterative formula of  the Green polynomials for $\GL_n(\mathbb F_q)$. Thus
vertex operators provide a practical
algorithm to compute the irreducible characters.

This article is divided into five sections.
After introduction, we review
colored partitions and construct tensor products of the ring of symmetric functions in the second section. In the third section, we first review Green's parametrization of conjugacy
classes of $\GL_n(\mathbb F_q)$ using the
Frobenius automorphism $F$, then we construct two Heisenberg algebras associated to $F$-orbits of the algebraically closed field $\overline{\mathbb F}_q$ and its character group. In section four, we realize the Grothendieck ring $R_G$ in terms of two versions of the Fock space equivariant under the Frobenius map $F$.
We introduce accordingly two sets of vertex operators
indexed by $F$-orbits in $\hat{\overline{\mathbb F}}_q$
and $F$-orbits in $\overline{\mathbb F}_q$. We then compute the commutators of two sets of Heisenberg generators.
In sections five we discuss how to compute all irreducible characters of $\GL_n(\mathbb F_q)$
using vertex operator. In particular, we give an alternative derivation of Green's degree formulas for all irreducible characters.

\section{Colored partitions and Hall algebra}

We first recall some basic notions about partition-valued functions or colored partitions following \cite{Mac} or \cite{St}.
A {\it partition} $\lambda$ is a finite sequence of descending nonnegative integers $(\lambda_1,\lambda_2,\ldots)$ such that $\lambda_1\geq\lambda_2\geq\ldots$, where
nonzero $\lambda_i$ are the {\it parts} of $\lambda$. The  {\it length} $l(\lambda)$ is the number of parts, and the {\it weight} $|\lambda|=\sum_{i}\lambda_i$, where We usually denote $\lambda\vdash |\lambda|$. By $\lambda'$ we mean the conjugate partition of $\la$ with parts $\la_i'=|\{j| \la_j\geq i\}|$. When the parts are not necessarily descending, $\lambda$ is called a {\it composition} of $n$,
denoted as $\lambda\vDash n$.
The set of partitions (resp. compositions) will be denoted as $\mathcal P$ (resp. $\mathcal C$).
Sometimes a partition $\lambda$ is arranged in the ascending order: $\lambda=(1^{m_1(\la)}2^{m_2(\la)}\ldots)$ with $m_i(\la)$ being the multiplicity of part $i$ in $\lambda$.
For $\lambda\in\mathcal P$ we will need the following notations:
\begin{align*}
z_\lambda&=\prod_{i\geq1}i^{m_i(\lambda)}m_i(\lambda)!\\
n(\lambda)&=\sum_{i=1}^{l(\lambda)}(i-1)\lambda_i.
\end{align*}

For a set $X=(x_1,x_2,\ldots)$, a partition valued function $\boldsymbol{\mu}:X\to\mathcal{P}$ is specified by
writing $\boldsymbol{\mu}=(\mu^{(x_1)},\mu^{(x_2),},\ldots)$ with each $\mu^{(x_i)}\in\mathcal P$. Alternatively $\boldsymbol{\mu}$ is also called a {\it colored partition} \cite{IJS}
indexed by the color set $X$, where $\mu^{(i)}$ is colored by $x_i$. In this paper, we will define
the weight of the colored partition by assigning
a degree to each color, namely, the weight
is defined by $\|\bmu\|=\sum_{x_i\in X}d(x_i)|\mu^{(x_i)}|$. Of course, we will sometimes let
$d(x_i)=1$ for all colors $x_i$.

Let $\Lambda$ be the ring of symmetric functions over $\mathbb Q$, and let $\Lambda_k=\Lambda\otimes_{\mathbb Z} k$ for
the field $k$.
$\Lambda_{\mathbb Q}$ has several bases: elementary symmetric functions $\{e_\lambda\}$, monomial functions  $\{m_\lambda\}$, homogeneous functions $\{h_\lambda\}$,
power-sum functions $\{p_\lambda\}$ and Schur functions $\{s_\lambda\}$. We will also work with the Hall-Littlewood symmetric functions
over $F=\mathbb Q(t)$. Let  $\{v_\lambda\}$  be any of them. Let $x^{(i)}=(x_{(i,1)},x_{(i,2)},\ldots)$ be the variables in the $i$-th ring $\Lambda_F$. For each $r$-colored partition $\boldsymbol{\lambda}$ we define
\begin{align}
 v_{\boldsymbol{\lambda}}=v_{\lambda^{(1)}}(x^{(1)})v_{\lambda^{(2)}}(x^{(2)})\ldots v_{\lambda^{(r)}}(x^{(r)}),
\end{align}
then $\{v_{\boldsymbol{\lambda}}\}$ forms a basis for the space $\Lambda_{F}^{\otimes r}$ \cite{IJS}. We also define
\begin{align}
z_{\boldsymbol\lambda}=z_{\lambda^{(1)}}\ldots z_{\lambda^{(r)}}.
\end{align}

Similarly, we can extend the inner product of $\Lambda_{\mathbb{Q}(t)}$ to $\Lambda_{\mathbb{Q}(t)}^{\otimes r}$ by
\begin{align}
\langle u_1\otimes\ldots\otimes u_r,v_1\otimes\ldots\otimes v_l \rangle=\langle u_1, v_1\rangle\ldots\langle u_r, v_r\rangle.
\end{align}

Now let's describe conjugacy classes of $\GL_n(\mathbb F_q)$.
Let $\mathbb{F}_q$ be the field of $q$ elements, $\overline{\mathbb F}_q$ the algebraic closure of $\mathbb F_q$. Let $F:x\mapsto x^{q}$ be the Frobenius automorphism of
$\overline{\mathbb F}$ over $\mathbb F_q$, the set $\mathbb F_{q^n}=\overline{\mathbb F}_q^{F^n}=\{{F^{n}(x)=x|x\in\bar{\mathbb F}_q}\}$ is the unique field extension of $\mathbb F_q$ of degree $n$. Let $M$\,(resp. $M_n$) denote the multiplicative group of $\bar{\mathbb F}_q$\,(resp. $\bar{\mathbb F}_{q^n}$) so that $M_n=M^{F^n}$ and $|M_n|=q^n-1$.

Let $M_n^*$ be the character group of $M_n$ consisting of linear characters of $M_n$. If $m$ divides $n$, the norm map $N_{n,m}:M_n\rightarrow M_m$ defined by $N_{n,m}(x)=x^{(q^n-1)/(q^m-1)}=\prod_{i=0}^{\frac{m}{n}-1}F^{mi}x$ is a surjective homomorphism. We can regard $M_m$ as a subgroup of ~$M_n$, and $M_m^*$ is embedded in $M_n^*$ by the transpose of the norm homomorphism $N_{n,m}$. Let $M^*=\bigcup_{n\geq1}M_n^*$ be the direct limit of the $M_n^*$, the Frobenius map $F:\xi\to\xi^{q}$ acts on $M^*$, and $M^*_n$ is the $F^n$-invariant space of $M^*$ for any $n\geqslant1$. We define a pairing of $M_n^*$ with $M_n$ by $\langle{\xi},{x}\rangle_n=\xi(x)$ for $\xi\in M_n^*$ and $x\in M_n$. If $m$ divides $n$ and $\xi\in M^*_m$, $x\in M_m$, we have $\langle{\xi},{x}\rangle_n=(\langle{\xi},{x}\rangle_m)^{\frac{n}{m}}$.

Let $\Phi$ (resp. $\Phi_n$) denote the set of {\it $F$-orbits} in $M$ (resp. $M_n$). Each orbit of $\Phi$ is of the form $\{x,x^q,\ldots,x^{q^d-1}\}$, where $x^{q^d}=x$. Let $f=\prod_{i=0}^{d-1}(t-x^{q^{i}})$ be the (irreducible) minimal polynomial of $x$, where each distinct root of $f$ belongs to the same orbit, and the degree $d(f)=d$ is exactly the cardinality of the orbit. Therefore we can also use $f=\prod_{i=0}^{d-1}(t-x^{q^{i}})$ to denote
the orbit, and let
$\Phi$ denote the set of monic irreducible polynomials over $\mathbb F_q$ excluding $t$. Clearly, $\Phi_n=\{f\in\Phi\mid \,d(f)| n\}$. See \cite{Isa} for details.

Similarly, let $\Phi^*$ (resp. $\Phi_n^*$) denote the set of $F$-orbits in $M^*$ (resp. $M_n^*$). For $\phi\in\Phi^*$, we denote by $d(\phi)$ its cardinality. Clearly, $\Phi_n^*=\{\phi\in\Phi^*|\,d(\phi)|n\}$.

We can define an $F$-invariant pairing
$M$ and $M^*$ or rather a pairing between $\Phi^*$ and $\Phi$ as follows. For
orbits $\phi\in \Phi^*$ and $f\in\Phi$, fixing an $x\in f$ we define
\begin{equation}\label{e:pairing1}
(\phi, f)_n=\frac1{d(\phi)}\sum_{\xi\in \phi}\langle\xi, x\rangle_n
\end{equation}
where $\xi$ runs through the orbit $\phi$ and $d(\phi)|n$ and $d(f)|n$. Note that $(\phi, f)_n=0$ unless $d(\phi), d(f)|n$. We remark that 
\eqref{e:pairing1} is independent from the choice of $x\in f$. This is because if $x'\in f$ then
$x=F^k(x')$ for some $k$, then $\langle\xi, x'\rangle_n=\langle\xi, F^kx\rangle_n= \langle F^{k}\xi, x\rangle_n$ and $F^k\xi$ also runs through the orbit of $\xi$.

This pairing is also symmetric in the sense that one also has that
\begin{equation}\label{e:pairing2}
(\phi, f)_n=\frac1{d(f)}\sum_{x\in f}\langle\xi, x\rangle_n
\end{equation}
summed over $x\in$ orbit $f$ and for a fixed $\xi\in \phi$ and $d(\phi), d(f)| n$. 
In fact, $\langle\xi^i, x\rangle=\langle\xi, x^i\rangle$ and we note that
$\sum_{i=0}^{n-1}\langle \xi^{q^i}, x\rangle_n=\frac{n}{d([\xi])}\sum_{i=0}^{d([\xi])-1}\langle\xi^{q^i}, x\rangle_n$ by the cyclic property of $M^*_n$. Similarly $\sum_{i=0}^{n-1}\langle \xi, x^{q^i}\rangle=\frac{n}{d([x])}\sum_{i=0}^{d([x])-1}\langle \xi, x^{q^i}\rangle_n$. Here $[\xi]$ or $[f]$ refer to the $F$-orbit containing $\xi$ or $f$.

For $g\in G_n$ the $n$-dimensional vector space $\mathbb F^n_q$ is canonically a $\mathbb F_q[t]$-module called $V_g$ with the action $t.v=gv$ for all $v\in \mathbb F^n_q$.
Two elements $g$ and $h$  are conjugate in $G_n$ if and only if their corresponding ${\mathbb F_q}[t]$-modules $V_g\simeq V_h$. Since ${\mathbb F_q}[t]$ is a principle ideal domain, $V_g$ is a direct sum of cyclic modules of the form ${\mathbb F_q}[t]/{(f)}^m$, where $m\geq1$, $f\in\Phi$. For example, the identity of $G_n$ corresponds to the
${\mathbb F_q}[t]$-module $V_1=({\mathbb F_q}[t]/(t-1))^{\oplus n}$.

\begin{prop}\cite{Mac}
The conjugacy classes in $G_n$ are parameterized by  partition valued functions $\boldsymbol{\mu}:\Phi\to\mathcal{P}$ such that
\begin{align*}
\|\bmu\|=\sum_{f\in\Phi}d(f)|\boldsymbol{\mu}(f)|=n.
\end{align*}
Explicitly, the class $c_{\boldsymbol{\mu}}\ni g$ gives rise to
\begin{align*}
V_g=V_{\boldsymbol{\mu}}=\bigoplus_{f\in\Phi}\bigoplus_{i\geq1}{\mathbb F_q}[t]/{(f)}^{\mu_i{(f)}}
\end{align*}
which corresponds to the partition valued function  $\boldsymbol{\mu}: f\mapsto {\boldsymbol{\mu}}(f)=(\mu_1(f),\mu_2(f),\ldots)$.
\end{prop}

For this reason, we will refer these $\Phi$-colored partitions such that $\|\bmu\|=\sum_{f\in\Phi}d(f)|\boldsymbol{\mu}(f)|$
as {\it weighted partition-valued functions} $\Phi\rightarrow \mathcal P$.

We will use the $q$-integer $[n]=1+q+\cdots+q^{n-1}$ in this paper. Define
\begin{align}
a_{\lambda}(q)=q^{|\lambda|+2n(\lambda)}b_\lambda(q^{-1}),
\end{align}
where $b_{\lambda}(t)=(1-t)^{l(\lambda)}\prod_{i\geqslant 1}[m_i(\la)]_t!$, $  [n]_t!=[n]_t[n-1]_t\ldots [1]_t$. One can also define the Gauss number $\begin{bmatrix} m\\ n\end{bmatrix}=[m]!/[n]![m-n]!$.
Let $V_{(f)}$ denote the $f$-primary component of $V_{\boldsymbol{\mu}}$: $V_{(f)}\cong\bigoplus_i\mathbb F_q[t]/{(f)}^{\mu_i(f)}$. The localization $\mathbb F_q[t]_{(f)}$ is a discrete valuation ring with $(\mathbb F_q[t])_f=\mathbb F_q[t]/(f)$ as  the residue field, $V_{(f)}$ is a finite $\mathbb F_q[t]_{(f)}$-module of type $\boldsymbol{\mu}(f)$, then $|\mathrm{Aut}(V_{f})|=a_{\boldsymbol{\mu}(f)}(q_f)$, where $q_f=\mathrm{Card}((\mathbb F_q)_f)=q^{d(f)}$ \cite[Chap. 2]{Mac}. For any $g\in G_n$, the automorphisms of $V_g$ is just the elements $h\in G_n$ commuting with $g$, hence the order of centralizer of $g$ in $G_n$ is
\begin{align}
a_{\boldsymbol{\mu}}=\prod_{f\in\Phi}a_{\boldsymbol{\mu}(f)}(q_f).
\end{align}

Let $H(\mathbb F_q[t]_{(f)})$ be the free $\mathbb Z$-module with the basis of finite modules $u_{\lambda}=V_{\lambda}=\oplus_i\mathbb F_q[t]/(f)^{\lambda_i}$ indexed by all partitions $\lambda\in\mathcal P$. $H(\mathbb F_q[t]_{(f)})$ becomes the Hall algebra under the multiplication:
\begin{equation}
u_{\mu}v_{\mu}=\sum_{\lambda}G^{\lambda}_{\mu\nu}u_{\lambda}
\end{equation}
where $G^{\lambda}_{\mu\nu}$ is the number of the extensions $V$ of $V_{\mu}$ by $V_{\nu}$ or
the short exact sequences
\begin{equation}
0\longrightarrow V_{\mu}\longrightarrow V\longrightarrow V_{\nu}\longrightarrow 0
\end{equation}
such that $V/V_{\mu}\simeq V_{\nu}$ (then $V$ is of type $\lambda$).
We recall the following result on Hall algebras.

\begin{prop}\cite[Chap. II.4]{Mac}\label{P:Hall}
The Hall algebra $H(\mathbb F_q[t]_{(f)})$ is generated as a $\mathbb Z$-algebra by the algebraically independent elements $u_{(1^r)}$ and the structure constants $G_{\mu\nu}^{\lambda}=0$ unless $|\lambda|=|\mu|+|\nu|$ and
are uniquely determined by
\begin{equation}
G^{\lambda}_{\mu(1^r)}(q_f)=q_f^{n(\lambda)-n(\mu)-n(1^r)}
\prod_{i\geq 1}\begin{bmatrix}\lambda_i'-\lambda'_{i+1}\\ \lambda_i'-\mu_i'\end{bmatrix}_{q_f^{-1}},
\end{equation}
where $\lambda/\mu$ is a vertical $r$-strip and $\la'_i$ are the $i$-th parts of the conjugate partition $\lambda'$ of $\lambda$.
\end{prop}

\section{Construction of two Heisenberg systems}

Let $R(G_n)$ be the space of complex valued class functions on $G_n$ equipped with the Hermitian scalar product (see for example \cite{Se}):
\begin{align*}
\langle u,v\rangle_{G_n}=\frac{1}{|G_n|}\sum_{g\in G_n}u(g)\overline{v(g)}, \qquad u, v\in R(G_n).
\end{align*}
$R(G_n)$ is a complex vector space of $\dim R(G_n)=|\mathcal P_n(\Phi)|$, the number of weighted partition valued functions $\bl: \Phi\rightarrow \mathcal P$ such that $||\bl||=\sum_{f\in\Phi}d(f)|\bl(f)|=n$. It is cautioned that this number is not that of the partition-valued functions from $\Phi\rightarrow \mathcal P$
of the usual weight $n$.

Introduce the graded space
\begin{align*}
R_G=\bigoplus_{n\geq 0}R(G_n),
\end{align*}
with the inner product induced from that of $R(G_n)$:
$\langle R(G_m), R(G_n)\rangle=\delta_{m, n}\langle R(G_n), R(G_n)\rangle_{G_n}$.
Let $P_{(n_1, \ldots, n_k)}$ be the parabolic subgroup of $G_n$ in type $(n_1, \ldots, n_k)\Vdash n$ consisting of block upper triangular matrices
\begin{equation*}
\begin{pmatrix} g_{11} & g_{12} & \dots & g_{1k}\\
 0 & g_{22} & \dots & g_{2k}\\
 \vdots & \vdots & \ddots &  \vdots \\
 0 & 0 & \dots & g_{kk}\end{pmatrix}
\end{equation*}
where $g_{ij}\in \mathrm{Mat}(n_i, n_j)$ and $g_{ii}\in \GL_{n_i}(\mathbb F_q)$.

We define a multiplication on $R_G$ as follows. For $u\in R(G_n),v\in R(G_m)$, the class function $u\circ v\in R(G_{n+m})$ is given by
\begin{align}
u\circ v=\mathrm{Ind}_{P_{(n, m)}}^{G_{n+m}}(u(g_{11})v(g_{22})),
\end{align}
where $P_{(n, m)}=\begin{pmatrix} G_n & *\\
 0 & G_m\end{pmatrix}$ is the parabolic subgroup of $G_{n+m}$ in type $(n, m)$. The parabolic subgroup $P_{(n, m)}$ consists of elements $g\in G_{n+m}$ fixing the flag $F$:
 \begin{equation}\label{e:flag}
 \qquad 0=V_0\subset V_1\subset V_2=V=\mathbb F_q^{n+m},
 \end{equation}
 where $V_1\simeq \mathbb F_q^n$ is embedded in $V_2$ canonically
 (spanned by the first $n$ unit vectors).
 It is easy to see that $R_G$ is an associative commutative algebra under the multiplication.

Let $\pi_{\bmu}$ be the characteristic function of the conjugacy class $c_{\bmu}$ defined by
\begin{align*}
\pi_{\bmu}(c_{\bl})=\begin{cases} 1, & \mathrm{if} \quad \bl=\bmu\\
0, & \mathrm{if}\quad \bl\neq \bmu
\end{cases}.
\end{align*}
Then the $\pi_{\bmu}$, where $\bmu:\Phi\to\mathcal{P}$ such that $\|\bmu\|<\infty$, form a $\mathbb{C}$-basis of $R_G$.

Let's consider the product $\pi_{\bmu}\circ\pi_{\bnu}$. Let $G_{n+m}=\bigcup_ig_iP_{(n, m)}$ be a left coset decomposition. Then for any $x\in G_{n+m}$
 \begin{equation}
 \pi_{\bmu}\circ\pi_{\bnu}(g)=\sum_i(\pi_{\bmu}\times \pi_{\bnu})(g_i^{-1}gg_i)
 \end{equation}
where the function $\pi_{\bmu}\times \pi_{\bnu}$  vanishes outside of $P_{(n, m)}$. The element $g_i^{-1}gg_i\in P_{(n, m)}$ if and only if $g$ fixes the flag $F$ \eqref{e:flag}, or $g_iF$ is a flag of
submodules of the $\mathbb F_q[t]$-module $V_g$. Write
$W_i=tV_i$ and
\begin{equation*}
g_i^{-1}gg_i=\begin{pmatrix}h_{11} & h_{12}\\
0 & h_{22}\end{pmatrix}
\end{equation*}
then the modules $W_1\simeq V_{h_{11}}$ and $W_2/W_1\simeq V_{h_{22}}$. As $\pi_{\bmu}$ and $\pi_{\bnu}$ are
the characteristic functions at the conjugacy classes
$c_{\bmu}$ and $c_{\bnu}$, $(\pi_{\bmu}\times \pi_{\bnu})(g_i^{-1}gg_i)=1$ if and only if the flag $g_iF$ is of
of type ($\bmu$, $\bnu$).
Therefore by Prop. \ref{P:Hall} we have the following result

\begin{prop}\cite{jaG}\cite[Chap.IV.3]{Mac}
The space $R_G$ is a tensor product of the Hall algebras:
\begin{equation}
R_G=\bigotimes_{f\in \Phi}H(\mathbb F_q[t]_{(f)})\otimes_{\mathbb Z}\mathbb C
\end{equation}
which is spanned by the characteristic functions
$\pi_{\bmu}$ with the multiplication given by
\begin{align*}
\pi_{\bmu}\circ\pi_{\bnu}=\sum_{\bl}{\bf G}^{\bl}_{{\bmu}{\bnu}}\pi_{\bnu}
\end{align*}
where ${\bf G}^{\bl}_{{\bmu}{\bnu}}=\prod_{f\in\Phi}G^{\bl(f)}_{{\bmu(f)}{\bnu(f)}}(q^{d(f)})$ and $G^{\la}_{\mu\nu}(q)$ is the polynomial in $q$ uniquely determined by $$G^{\la}_{\mu(1^m)}(q)=q^{n(\la)-n(\mu)-n((1^m))}\prod_{i\geq1}\begin{bmatrix}\la^\prime_i-\la^\prime_{i+1}\\\la^\prime_i-\mu^\prime_i\end{bmatrix}_{q^{-1}}
$$
for $\lambda/\mu$ being a vertical $m$-strip.
In particular, $\pi_0$ is the identity element of
$R_G$.
\end{prop}

We now identify $R_G$ to the Fock space of
symmetric functions.
Let $X_{i,\phi}(i\geq1,\phi\in\Phi^*)$ be independent variables over $\mathbb{C}$ associated with $\phi$. Let $p_n(\phi)$ be the $r$-th power sum function $p_n(\phi)=p_n(X_{1,\phi},X_{2,\phi},\ldots)=\sum_iX_{i,\phi}^n.$
Let
\begin{align}
\mathcal{F}_G=\bigotimes_{\phi\in\Phi^*}\mathbb{C}[p_n(\phi):n\geq1]\simeq \bigotimes_{\phi\in\Phi^*}\Lambda_{\mathbb C}
\end{align}
be the polynomial algebra generated by the $p_n(\phi)$. The algebra $\mathcal{F}_G$ is a graded algebra with degree given by $\mathrm{deg}(p_n(\phi))=n\cdot d(\phi)$.
Note that $\mathcal F_G$ is the unique level one irreducible representation of the infinite-dimensional Heisenberg Lie algebra $\widehat{\mathfrak h}_{\hat{\overline{\mathbb F}}_q}$ spanned by the $p_n(\phi)$ and the central element $C$ subject to the relation:
\begin{align}
[p_m(\phi), p_n(\phi')]=m\delta_{m, -n}\delta_{\phi, \phi'}C.
\end{align}

For each weighted partition valued function $\tilde{\bl}:\Phi^*\to\mathcal{P}$ such that
\begin{align*}
\|\tilde{\bl}\|=\sum_{\phi\in\Phi^*}d(\phi)|\tilde{\bl}(\phi)|<\infty
\end{align*}
we define the power-sum symmetric function $
p_{\tilde{\bl}}=\prod_{\phi\in\Phi^*}p_{\tilde{\bl}(\phi)}(\phi).
$
Then the space $\mathcal{F}_G$ inherits the inner product (sesquilinear bilinear form) over $\mathbb C$ given by $\langle p_m(\phi), p_n(\phi')\rangle=m\delta_{m, n}\delta_{\phi, \phi'}$ and more generally
\begin{align}\label{def:inn}
\langle p_{\tilde{\bl}},p_{\tilde{\bmu}} \rangle=\delta_{\tilde{\bl}\tilde{\bmu}}z_{\tilde{\bl}}
\end{align}
where $z_{\tilde{\bl}}=\prod_{\phi\in\Phi}z_{\tilde{\bl}(\phi)}$. Componentwise, it is given by
$\langle p_{\la}(\phi_1),p_{\mu}(\phi_2)\rangle=z_\la\delta_{\la\mu}\delta_{\phi_1,\phi_2}$
for $\la, \mu\in\mathcal P$ and $\phi_1, \phi_2\in \Phi^*$.



For an orbit $\phi\in\Phi^*$ and any $\xi\in\phi$, define
\begin{align*}
\hat{p}_n(\xi)=\begin{cases} p_{n/d(\phi)}(\phi) & \mathrm{if} \;d(\phi)|n,\\0 &\mathrm{otherwise},
\end{cases}
\end{align*}
Note that $\hat{p}_n(\xi)=\hat{p}_n(\xi')$ if $\xi, \xi'$ are in the same orbit $\phi$, hence $\hat{p}_n(\xi)\in\mathcal{F}_G$ with degree $n$ and $\hat{p}_n(\xi)$=0 unless $\xi\in M^*_n$. Then we have that
\begin{equation}\label{e:comhat}
\langle \hat{p}_n(\xi), \hat{p}_m(\xi')\rangle=\begin{cases}\frac{n}{d(\phi)}\delta_{n, m} & \mbox{$\xi, \xi'$ in the same orbit $\phi$}\\
0 & \mbox{otherwise}
\end{cases}
\end{equation}

 We introduce another basis indexed by partitions colored by $\Phi$ via the discrete Fourier transform. For
 each $x\in M$, we define
\begin{align}\label{rel:org}
\hat{p}_n(x)=\begin{cases}
(-1)^{n-1}(q^n-1)^{-1}\sum_{\xi\in M^*_n}\overline{\langle \xi,x\rangle}_n\hat{p}_n(\xi) &\mathrm{if}\;x\in M_n\\
0 &\mathrm{if}\;x\notin M_n.
\end{cases}
\end{align}
Then $\hat p_n(x)=\hat p_n(y)$ if $x, y$ are in the same $F$-orbit, and $\hat{p}_n(x)$ are also homogeneous elements of degree $n$. By the second orthogonality relation of irreducible characters of
the finite abelian group $M_n$, it is easy to check that
\begin{align}\label{org2}
\hat{p}_n(\xi)=(-1)^{n-1}\sum_{x\in M_n}\langle \xi,x\rangle_n\hat{p}_n(x)
\end{align}
for $\xi\in M^*_n$. For $x\in M$, let $f$ be the minimal polynomial for $x$ over $k$, i.e. $f$ represents the $F$-orbit of $x$. For each $n\geqslant1$, we similarly define for any orbit $f\in \Phi$
\begin{align}\label{def:f,x}
p_n(f)=\hat{p}_{nd}(x), \qquad\qquad d(f)=d
\end{align}
for some $x\in f$. Note that $deg(p_n(f))=nd(f)$.

\begin{prop}\label{compute0}
For any two polynomials $f_1,f_2\in\Phi$, we have
\begin{align}\label{ref:f,f}
\langle p_n(f_1),p_m(f_2)\rangle=n(q_{f_1}^n-1)^{-1}\delta_{n,m}\delta_{f_1,f_2}.
\end{align}
\end{prop}
\begin{proof}
Let $x$ be in orbit $f_1$ with $d_1=d(f_1)$ and let $y$ be in $f_2$ with $d_2=d(f_2)$. Denote the orbit of the character $\xi\in M^*$ by $[\xi]$. 
It follows from \eqref{e:comhat} that
\begin{align*}
\langle p_n(f_1), p_m(f_2)\rangle&=\langle \hat{p}_{nd_1}(x),\hat{p}_{md_2}(y)\rangle
\\&=\frac{{(-1)}^{nd_1+md_2}}{(q^{nd_1}-1)(q^{md_2}-1)}\sum_{\xi\in M^*_{nd_1}}\sum_{\xi^\prime\in M^*_{md_2}}\overline{{\langle\xi,x \rangle}}_{nd_1}\langle \xi^\prime,y\rangle_{md_2}\langle \hat{p}_{nd_1}(\xi),\hat{p}_{md_2}(\xi^\prime)\rangle
\end{align*}
where $\xi, \xi'$ must be in the same orbit $[\xi]=\phi$ and $nd_1=md_2$.
Therefore the above is simplified to
\begin{align}\nonumber
&\frac{nd_1\delta_{nd_1, md_2}}{(q^{nd_1}-1)^2}\sum_{\xi, \xi'\in M_{nd_1}^*, [\xi]=[\xi']}\frac1{d([\xi])}\overline{\langle\xi,x \rangle}_{nd_1}\langle \xi^\prime,y\rangle_{nd_1}\\ \label{e:inner}
&=\frac{\delta_{nd_1, md_2}}{(q^{nd_1}-1)^2}\sum_{\xi\in M_{nd_1}^*}\frac{nd_1}{d([\xi])}\overline{\langle\xi,x \rangle}_{nd_1}(\sum_{j=0}^{d([\xi])-1}\langle \xi,y^{q^j}\rangle_{nd_1}),
\end{align}
where we have used the fact that $\langle\xi^k, y\rangle = \langle \xi, y^k\rangle$ for
$0\leq k\leq d([\xi])-1$.

For any positive integer $k$, we know that the Frobenius map $F$ acts on $M^*_k$ additively such that $\{F^i| 0\leq i\leq k-1\}$ is isomorphic to the cyclic group $\mathbb Z_k$.
 Now for any $\xi\in M_k^*$, the image of the $F$-orbit $[\xi]$ divides $\mathbb Z_k$ into
$k/d([\xi])$ equal-cardinality subsets, in other words, we have that
\begin{equation}\label{e:F-id}
\sum_{j=0}^{k-1}\xi(y^{q^j})=\sum_{j=0}^{k-1}\xi^{q^j}(y)=\frac{k}{d([\xi])}\sum_{j=0}^{d([\xi])-1}\xi(y^{q^j}).
\end{equation}
Therefore \eqref{e:inner} can be rewritten as
\begin{align*}
&\frac{\delta_{nd_1, md_2}}{(q^{nd_1}-1)^2}\sum_{\xi\in M_{nd_1}^*}\sum_{j=0}^{nd_1-1}\overline{\langle\xi,x \rangle}_{nd_1}\langle \xi,y^{q^j}\rangle_{nd_1}\\
&=\frac{\delta_{nd_1, md_2}}{(q^{nd_1}-1)^2}\sum_{j=0}^{nd_1-1}\sum_{\xi\in M_{nd_1}^*}\overline{\langle\xi,x \rangle}_{nd_1}\langle \xi,y^{q^j}\rangle_{nd_1}\\
&=\frac{\delta_{nd_1, md_2}}{q^{nd_1}-1}\sum_{j=0}^{nd_1-1}\delta_{x, y^{q^j}}=\frac{n}{q^{nd_1}-1}\delta_{f_1, f_2}\delta_{n, m},
\end{align*}
where the last equality is also due to the same reason explained above \eqref{e:F-id}.\end{proof}

We now regard the polynomial $p_r(f)$ as the power sums of a sequence of variables $Y_{i,f}$, that is $p_r(f)
=\sum_{i\geq 1}Y_{i,f}^r$, each variable has degree $d(f)$. Then the degree of $p_r(f)$ is $r\cdot d(f)$. For each $\phi\in\Phi^*$, $p_r(\phi)$ is a linear combination of $p_n(f)$, $f\in\Phi$, then we can also regard $\mathcal F_G$ as the space
\begin{align}\label{e:FG}
\mathcal{F}_G=\bigotimes_{f\in \Phi}\mathbb{C}[p_n(f):n\geqslant1]=\bigotimes_{f\in \Phi}\Lambda_{\mathbb C}(f),
\end{align}
where $\Lambda_{\mathbb C}(f)=\mathbb C[p_1(f), p_2(f), \ldots]\cong \Lambda_{\mathbb C}$ is the space of
symmetric functions generated by the power-sum symmetric functions $p_n(f)$.

The Fock space is also the unique level one irreducible representation of the infinite dimensional Heisenberg Lie algebra $\widehat{\mathfrak{h}}_{\overline{\mathbb F}_q}$ generated by the $p_n(f)$ and the central element $C$ with
the relation
 \begin{align}\label{e:Heisenberg2}
[p_m(f), p_n(f')]=\frac{m}{q^{d(f)m}-1}\delta_{m, -n}\delta_{f, f'}C, \qquad\quad (m>0)
\end{align}
The Fock space $\Lambda_{\mathbb C}(f)$ is a graded ring with the natural gradation defined by the degree $d(p_r(f))=rd(f)$.

The unique irreducible representation
is realized as follows. For each $f\in\Phi$, the action $p_n(f):\mathcal{F}_G\to \mathcal{F}_G$ is regarded as multiplication operator of degree $nd(f)$, and the adjoint operator is the differential action $p_n^*(f)=\frac{n}{q_f^n-1}\frac{\partial}{\partial p_n(f)}$ of degree $-nd(f)$. Note that $*$ is $\mathbb{C}$-linear
and anti-involution, so
\begin{align}
\langle p_n(f)u,v\rangle=\langle u,p_n^*(f)v\rangle
\end{align}
for $u,v\in\mathcal{F}_G$. Then for $f\in\Phi$ and partitions $\la$ and $\mu$, we have
\begin{align}\label{rel:p,f}
\langle p_\la(f),p_\mu(f)\rangle=z_\la(q_f^{-1})q_f^{-|\la|}\delta_{\la\mu}.
\end{align}

Next we claim that
\begin{equation}\label{e:commutator1}
    \langle p_m(\phi), p_n(f)\rangle=(-1)^{md(\phi)-1}\frac{md(\phi)(\phi, f)}{q_{\phi}^{m}-1}\delta_{md(\phi),nd(f)}
\end{equation}
where $(\phi, f)=(\phi, f)_{md(\phi)}$.
In fact, let $x\in$ orbit $f\in\Phi$
and $\xi\in$ orbit $\phi\in \Phi^*$, then
\begin{align*}
    \langle p_m(\phi), p_n(f)\rangle&=
    \langle \hat{p}_{md(\phi)}(\xi),\hat{p}_{nd(f)}(x)\rangle\\&=\langle \hat{p}_{md(\phi)}(\xi), \frac{(-1)^{nd(f)-1}}{q^{nd(f)}-1}\sum_{\xi'\in M_{nd(f)}^*}\overline{\langle \xi', x\rangle}_{nd(f)} \hat{p}_{nd(f)}(\xi')\rangle\\
    &=\frac{(-1)^{md(\phi)-1}m\sum_{\xi\in\phi}\langle \xi, x\rangle_{md(\phi)}}{q^{md(\phi)}-1}\delta_{md(\phi), nd(f)},
\end{align*}
which shows \eqref{e:commutator1} by recalling the pairing $(\phi, f)$ from \eqref{e:pairing1}.
Equivalently,
\begin{equation}\label{e:commutator2}
   [p_{m}^*(\phi),p_{n}(f)]=(-1)^{md(\phi)-1}\frac{md(\phi)(\phi, f)}{q_{\phi}^{m}-1}\delta_{md(\phi),nd(f)}
\end{equation}
As a consequence, for the orbit $f_1=t-1$ of the identity element $(\phi, f_1)_{n}=1$, we have that
\begin{equation}\label{e:comm}  
[p_{m}^*(\phi),p_{n}(f_1)]=(-1)^{n-1}\frac{n}{q_{\phi}^m-1}\delta_{md(\phi),n}
\end{equation}

\section{Vertex operators for Schur and Hall-Littlewood functions}


In the realization of $\mathcal F_G$ as the second
Fock space \eqref{e:Heisenberg2} spanned by
the symmetric algebra $Sym[p_n(f)'s]$, we note that
this space is canonically isomorphic to $\Lambda_{\mathbb C}(f)=Sym[p_n(f)'s]$
by identifying $h_{-n}(f)$ with $p_n(f)$.
Then the action is given by $c=1$, $h_{-n}(f)$ ($n>0$)
acts as the multiplication operator $p_n(f)$,
and $h_{n}(f)$ ($n>0$) acts as the differential operator
$\frac{n}{q_f^n-1}\frac{\partial}{\partial p_n(f)}$ for any $f\in\Phi$.

For each $f\in\Phi$, following \cite{Jing2} we introduce the vertex operator $Q(f; z)$ as the
following product of exponential operators: $\mathcal{F}_G\to \mathcal{F}_G[[z,z^{-1}]]$
\begin{align}\label{e:VO1}
Q(f; z)&=\mbox{exp} \left( \sum\limits_{n\geq 1} \dfrac{q_f^n-1}{n}p_n(f)z^{n} \right) \mbox{exp} \left( -\sum \limits_{n\geq 1} q^{-n}_f\frac{\partial}{\partial p_n(f)}z^{-n} \right)=\sum_{n\in\mathbb Z}Q_n(f)z^{n},
\end{align}
where $Q_n(f)\in End(\mathcal{F}_G)$.
We remark that $Q(f; z)$ was the vertex operator $Q(q_fz)$ in \cite{Jing2}, therefore
\begin{equation}
Q_n(f).1=q_f^{n}\sum_{\lambda\vdash n}\frac{p_{\lambda}(f)}
{z_{\lambda}(q^{-1}_f)}.
\end{equation}
which is the Hall-Littlewood polynomial associated to one row partition $(n)$ up to the factor $q_f^{-n}$. We recall two propositions from \cite{Jing2} as follows.

\begin{prop}\cite{Jing2} \label{P:HLv} For each
$f\in\Phi$ the set of vectors
$Q_{\la}(f).1=Q_{\la_1}(f)\ldots Q_{\la_l}(f).1$, where $\la=(\la_1,\ldots,\la_l)\in\mathcal P$,
forms
an orthogonal
basis of the ring $\Lambda_{\mathbb C(q_f)}(f)$ of symmetric functions over the field $\mathbb Q(q_f)$  generated by the power-sums $p_1(f), p_2(f), \ldots$, where $1$ is the vacuum vector.
Explicitly one has that
\begin{equation}\label{e:norm}
    \langle Q_{\lambda}(f).1, Q_{\mu}(f).1\rangle =\delta_{\lambda,\mu}
    q^{|\lambda|}_fb_{\lambda}(q^{-1}_f).
\end{equation}
\end{prop}
\begin{proof} Let $\tilde{p}_n(f)=q_f^{n/2}p_n(f)$, then
$Q(f; z)$ can be rewritten as
\begin{align*}
Q(f; z)&=\mbox{exp} \left( \sum\limits_{n\geq 1} \dfrac{1-q_f^{-n}}{n}\tilde{p}_n(f)(q^{1/2}_fz)^{n} \right) \mbox{exp} \left(-\sum \limits_{n\geq 1} \frac{\partial}{\partial \tilde{p}_n(f)}(q^{1/2}_fz)^{-n} \right)
\end{align*}
which is exactly the vertex operator $Q(q_f^{1/2}z)$ introduced in \cite{Jing2} associated with the Heisenberg Lie algebra $[\tilde{h}_m(f), \tilde{h}_n(f)]=m\delta_{m, -n}/(1-q^{-|m|}_f)$.
Then \eqref{e:norm} follows from
\cite[Prop. 2.20]{Jing2}.
\end{proof}
We remark that the vector
$q_f^{-|\lambda|}Q_{\lambda}(f).1$ was shown to realize the
Hall-Littlewood symmetric function (with $t=q^{-1}_f$)
associated with $\lambda$ in \cite{Jing2}, where the identification had used
the raising formula of Hall-Littlewood function \cite{Mac}
due to Littlewood's formulation.
In the following,
we give another proof
that $\Lambda_{\mathbb C}(f)$ is the Hall algebra using vertex operator calculus, which will identify $q_f^{-|\lambda|}Q_{\lambda}(f).1$ as the Hall-Littlewood polynomial associated with $\lambda$.
This also
helps us to identify the characteristic class function $\pi_{\bl}$ with the vector
$\prod_{f\in\Phi}Q_{\lambda(f)}$,
which is a new proof  of Green's results
\cite{jaG} that logically bypasses Hall's theorem. 

\begin{prop}\cite{Jing2}
For any $f\in \Phi$ and and any positive integer $m\geq 1$, we have the relation
\begin{align*}
Q_n(f)Q_{m+1}(f)-q_f^{-1}Q_{n+1}(f)Q_{m}(f)=q_f^{-1}Q_{m+1}(f)Q_n(f)-Q_m(f)Q_{n+1}(f).
\end{align*}
In particular, let $m=n$, it simplifies to the commutation relation:
\begin{align}\label{Q_m,m+1}
Q_{m}(f)Q_{m+1}(f)=q_f^{-1}Q_{m+1}(f)Q_m(f).
\end{align}
\end{prop}

For a fixed  $f\in \Phi$ and any partition $\lambda=(\lambda_1, \ldots, \lambda_l)$,
we write $Q_{\lambda}(f)=Q_{\lambda_1}(f)\ldots Q_{\lambda_l}(f)$. Then
for $\bl: \Phi\rightarrow \mathcal P$ such that $||\bl||=\sum_{f\in\Phi}d(f)|\bl(f)|<\infty$, we define the following two sets of orthogonal vectors in $\mathcal{F}_G$:
\begin{align}\label{e:cha}
Q_{\bl}&=\prod_{f\in\Phi}Q_{\bl(f)}(f).1,\\ \label{e:chb}
P_{\bl}&=q^{-||\bl||}\prod_{f\in\Phi}b_{\bl(f)}(q_f^{-1})^{-1}Q_{\bl(f)}(f).1
\end{align}
which are homogeneous polynomials in the $p_n(f)$ colored by $\Phi$ of degree $\|\bl\|$ and form dual bases in $\mathcal{F}_G$:
\begin{equation}
\langle P_{\bl}, Q_{\bmu}\rangle=\delta_{\bl, \bmu}.
\end{equation}

We now show that the vector $Q_{\bl}$ can be identified with the characteristic class function $\pi_{\bl}$ up to a power of $q$.
Let us introduce the half vertex operator $E(z, f): \mathcal{F}_G\to \mathcal{F}_G[[z,z^{-1}]]$ by
\begin{align*}
E(f; z)&=\mbox{exp} \left(\sum\limits_{n\geq 1} \dfrac{(-1)^{n-1}}{n}p_n(f)z^n\right)=\sum_{n\geq0}E_n(f)z^n,
\end{align*}
where the symmetric functions $E_n(f)$ are explicitly given by
\begin{equation}\label{R:P,Q}
E_n(f)=\sum_{\la\vdash \mathcal P_n}\frac{\epsilon({\la})}{z_{\la}}p_{\la}(f).
\end{equation}
Therefore
$E_n(f)$ is the elementary symmetric function $e_n(f)$ or the Hall-Littlewood symmetric function $P_{(1^n)}(f)$. The adjoint vertex operator under the
inner product $\langle p_{\lambda}(f), p_{\mu}(f)\rangle =z_{\la}(q_f^{-1})q_f^{-|\lambda|}\delta_{\la, \mu}$ is
\begin{align*}
E^*(f; z)&=\mbox{exp} \left(-\sum\limits_{n\geq 1} \dfrac{(-1)^n}{q_f^n-1}\frac{\partial}{\partial p_n(f)} z^n\right)=\sum_{n\geq0}E_n^*(f)z^n.
\end{align*}

\begin{thm}\label{T:EQ}
Let $\nu=(\nu_1,\ldots,\nu_l)$ be a partition and $r$ a positive integer $\leq l$.
Then for any $f\in \Phi$ we have that
\begin{align*}
E_r^*(f)Q_{\nu_1}(f)\cdots Q_{\nu_l}(f).1=\sum\limits_{(i_1,\ldots,i_r)\subset I}Q_{\nu-(e_{i_1}+\ldots+e_{i_r})}(f).1
\end{align*}
summed over all subset $(i_1,\ldots,i_r)$ of $I=\{1,2,\ldots,l\}$. Here $e_i=(\ldots,0,1,0,\ldots)$ are the unit vectors.
\end{thm}
\begin{proof} By using the commutation relations
\begin{align*}
E^*(f; w)Q(f; z)&=\mbox{exp}\left(\left[-\sum\limits_{n\geq 1}\dfrac{(-1)^n}{q_f^n-1}\frac{\partial}{\partial p_n(f)}w^n,  \sum\limits_{n\geq 1}\dfrac{q_f^n-1}{n}p_n(f)z^n\right]\right)Q(f; z)E^*(f; w)\\
            &=\exp\left(\sum_{n\geq1}\dfrac{(-1)^{n+1}(wz)^n}{n}\right)Q(f; z)E^*(f; w)\\
            &=(1+wz)Q(f; z)E^*(f; w).
\end{align*}
It then follows that
\begin{align*}
E^*(f; w)Q(z_1;f)\cdots Q(z_l;f).1=\prod_{i\geq 1}^l(1+wz_i)Q(f; z_1)\ldots Q(f; z_l).1
\end{align*}
Taking coefficient of $w^rz_1^{\nu_1}\ldots z_l^{\nu_l}(r\leq l)$ of both sides, we have
\begin{align*}
E^*_r(f)Q_{\nu_1}(f)\ldots Q_{\nu_l}(f).1=\sum\limits_{(i_1,\ldots,i_r)\subset I}Q_{\nu-(e_{i_1}+\ldots+e_{i_r})}(f).1
\end{align*}
\end{proof}


Now we fix $f\in\Phi$ and
simply write $Q_n$ for $Q_n(f)$ and $q$ for $q_f$ in the next two theorems, similarly $E_r^*$ for $E^*_r(f)$ as well.
We first consider the case of rectangular $\nu=(n^m)$.
\begin{lem}\label{l:action} Let $r\geq 0$ be an integer. The action of $E^*_r$ on $Q_{(n^m)}.1$ is zero unless $r\leq m$ and given by
\begin{equation*}
E_r^*Q_{(n^m)}.1=\begin{bmatrix} m\\ r\end{bmatrix}_{q^{-1}}Q_{(n^{m-r}(n-1)^r)}.1
\end{equation*}
\end{lem}
\begin{proof} By Theorem \ref{T:EQ}, the expansion of $E_r^*Q_{(n^m)}.1$ is the summation of all possible reductions of the indices by one in $Q_{(n^m)}=Q_nQ_n\ldots Q_n$ at $r$ factors (i.e changing $r$ factors of $Q_n$ to $Q_{n-1}$'s ). Using the commutation relation \eqref{Q_m,m+1} all these reduced products
can be brought to $Q_{(n^{m-r}(n-1)^r)}.1$, so we can write
\begin{equation}\label{e:action1}
E_r^*Q_{(n^m)}.1=K_r^m Q_{(n^{m-r}(n-1)^r)}.1
\end{equation}
We prove that $K_r^m$ is equal to the quantum binomial coefficient by induction on $r$. Clearly when $r=1$, repeatedly using \eqref{Q_m,m+1} we have that $Q_{n-1}Q_n^i=q^{-i}Q_n^iQ_{n-1}$, so $K_1^m=1+q^{-1}+\cdots+q^{-m+1}=[m]_{q^{-1}}$. 

Now consider the action of $E_r^*$ on $Q_n^m.1$. Taking the first factor $Q_n$ as a reference,
the index reduction of
$r$ factors can be divided into two groups: (i) the reduction does not involve with the first factor $Q_n$, so the action is the same as $E_r^*$ on $Q_{(n^{m-1})}$
and there are $K^{m-1}_{r-1}$ terms; (ii) If the reduction
includes changing first factor $Q_n$ to $Q_{n-1}$ and the rest of reductions are then among the latter $r-1$ factors which contributes $K^{m-1}_{r-1}$ terms. 

Here when $r-1$ factors are already reduced first
(this is the right order in the normal ordering process), only $n-r$ copies of $Q_n$'s left, so moving the first reduced factor $Q_{n-1}$ to the right produces a factor $q^{-n+r}$. So
we come to the relation:
$
K_{r}^{m}=K_{r}^{m-1}+q^{-m+r}K_{r-1}^{m-1}
$.
Therefore by the induction hypothesis, we have then
\begin{align*}
K_{r}^{m}&=K_{r}^{m-1}+q^{-m+r}K_{r-1}^{m-1}\\
&=\frac{[m-1]!}{[r]![m-r-1]!}+q^{-m+r}\frac{[m-1]!}{[r-1]![m-r]!}\\
&=\frac{[m-1]!}{[r]![m-r]!}\left([m-r]_{q^{-1}}+q^{-m+r}[r]_{q^{-1}}\right)\\
&=\frac{[m]!}{[r]![m-r]!},
\end{align*}
which proves the lemma.
\end{proof}

\begin{thm}\label{m}
Let $\nu$ be a partition
\begin{align*}
E^*_rQ_{\nu}.1=\sum_{\mu}\prod_{i\geq 1}\begin{bmatrix} \nu_i'-\nu_{i+1}'\\ \nu_i'-\mu_i'\end{bmatrix}_{q^{-1}}Q_{\mu}.1
\end{align*}
where the sum runs through all partitions $\mu$ such that $\nu/\mu$ are vertical $r$-strips.
\end{thm}
\begin{proof} Given the product $Q_{\nu}=\cdots Q_{i+1}^{m_{i+1}}Q_i^{m_i}Q_{i-1}^{m_{i-1}}\cdots $, the action of $E_r^*$
can be configured as follows.
It amounts to all possible reductions of indices by one in $r_1$ factors from $Q_1^{m_1}$, $r_2$ factors from $Q_2^{m_2}$, $\ldots$, and $r_i$ factors from $Q_i^{m_i}$ and so on, where $r_1+r_2+\cdots =r$. The
selections among each $Q_i\cdots Q_i=Q_i^{m_i}$ are order dependent, while the choices from different products $Q_i^{m_i}$ and $Q_j^{m_j}$ are independent.
It follows from Lemma \ref{l:action} that
the coefficient of $Q_{\mu}.1$ in $E^*_rQ_{\nu}.1$ is given by
\begin{equation}
\prod_{i\geq 1}K_{r_i}^{m_i}=\prod_{i\geq 1}\begin{bmatrix}m_i(\nu)\\r_i\end{bmatrix}_{q^{-1}}.
\end{equation}

In the action of $E^*_r$ on $Q_{\nu}.1$, the general summand $Q_{\nu-e_{i_1}-\cdots-e_{i_r}}.1$ is produced by
transforming $r_i$ copies of $Q_i$ into $Q_{i-1}$ as follows:
\begin{align*}
Q_{\nu}=Q_n^{m_n}Q_{n-1}^{m_{n-1}}\cdots Q_2^{m_2}Q_1^{m_1}\longrightarrow & Q_n^{m_n-r_n}Q_{n-1}^{r_n}Q_{n-1}^{m_{n-1}-r_{n-1}}Q_{n-2}^{r_{n-1}}\cdots
Q_2^{m_2-r_2}Q_1^{r_2}Q_1^{m_1-r_1}\\
&=Q_n^{m_n-r_n}Q_{n-1}^{m_{n-1}+r_n-r_{n-1}}\cdots Q_2^{m_2-r_2+r_1}Q_1^{m_1-r_1}=Q_{\mu}
\end{align*}
Therefore the exponents obey the relations: $m_i(\nu)-r_i+r_{i+1}=m_i(\mu)$ for all $i$. Note that the multiplicity $m_i(\nu)=\nu_i'-\nu_{i+1}'$. Thus
\begin{equation}
r_i=\sum_{j\geq i}(m_j(\nu)-m_j(\mu))=\nu_i'-\mu_i'.
\end{equation}

Therefore we have shown that the coefficient of $Q_{\mu}.1$ in $E^*_rQ_{\nu}.1$ is equal to
\begin{equation}
\prod_{i\geq 1}K^{m_i}_{r_i}=\prod_{i\geq 1}\begin{bmatrix}\nu_i'-\nu_{i+1}'\\ \nu_i'-\mu_i'\end{bmatrix}_{q^{-1}}.
\end{equation}
\end{proof}

Using the duality of $Q_{\lambda}$ and $P_{\lambda}$, we have the the following corollary immediately.
\begin{cor}\label{P,P,P}
Let $\mu$ be any partition, and let
$\la$ be the partition obtained from $\mu$ by adding the last column of $r$ boxes, then we have
\begin{align*}
P_\mu(f).1P_{(1^r)}(f).1=P_\la(f).1+\sum_{\nu}g^\nu_{\mu(1^r)}(q_f)P_\nu(f).1,
\end{align*}
where $g^{\nu}_{\mu(1^r)}(q_f)=\prod_{i\geq 1}\begin{bmatrix}\nu_i'-\nu_{i+1}'\\ \nu_i'-\mu_i'\end{bmatrix}_{q_f^{-1}}$ and $\nu$ runs through all partitions (except $\lambda$) such that $\nu/\mu$ are vertical $r$-strips.
\end{cor}

\begin{thm}\label{T:isom1} For fixed $f\in\Phi$,
let $\theta: H(\mathbb F_q[t]_{(f)})\otimes_{\mathbb Z}\mathbb{C}\to\Lambda_{\mathbb{C}}(f)=Sym(p_n(f)'s)$ be the $\mathbb{C}-$linear mapping defined by
\begin{align*}
\theta(u_\la)=q_f^{-n(\la)}P_\la(f).1.
\end{align*}
Then $\theta$ is a ring isomorphism.
\end{thm}
\begin{proof}
Note that $\{P_\la(f).1\}$ is a basis of $\Lambda_{\mathbb{C}}(f)$ by Prop. \ref{P:HLv}, so $\theta$ is a linear isomorphism. Since $H(\mathbb F_q[t]_{(f)})$ is freely generated by $u_{(1^r)}$ by Prop. \ref{P:Hall}, we may define a ring homomorphism
\cite{Mac} $\theta^\prime: H(\mathbb F_q[t]_{(f)})\otimes_{\mathbb {Z}}\mathbb{C}\to\Lambda_{\mathbb{C}}(f)$ by $\theta^\prime(u_{(1^r)})=q_f^{\frac{r(r-1)}{2}}e_r$. We will show $\theta=\theta^\prime$ which completes the proof.

We proceed by induction, assuming the result is true for all $\nu$ with $|\nu|<|\la|$. If $\la\neq0$, and let $\mu$ be a partition obtained from $\la$ by deleting last column, it is clear $\mu$ has length $|\mu|<|\la|$. Suppose that this last column has $m$ elements, we have
\begin{align*}
u_{\mu}u_{(1^m)}=u_\la+\sum_{\nu<\la}G_{\mu(1^m)}^\nu u_\nu.
\end{align*}
Applying $\theta^\prime$ to both sides of above equation, $\theta^\prime(u_\nu)=q_f^{-n(\nu)}P_\nu(f)(\nu<\la)$ by inductive hypothesis and Corollary \ref{P,P,P}, we have $\theta^\prime(u_\la)=\theta(u_\la)$.
\end{proof}
Since $u_\mu u_\nu=\sum_\la G^\la_{\mu\nu}(q_f)u_\la$ by Theorem \ref{T:isom1}, we have
\begin{align}\label{T,g}
G^\la_{\mu\nu}(q_f)=q_f^{n(\mu)+n(\nu)-n(\la)}g^\la_{\mu\nu}(q_f).
\end{align}

Then the following result is immediate from Theorem \ref{T:isom1}.
\begin{thm}\label{thm:ch}
The characteristic map $\mathrm{ch}:R_G\to\mathcal{F}_G$,
$$\pi_{\bmu}\mapsto \prod_{f\in\Phi}q_f^{-n(\bmu(f))}P_{\bmu(f)}.1
=\prod_{f\in\Phi}q_f^{-|\bmu(f)|-n(\bmu(f))}b_{\bmu(f)}(q_f^{-1})^{-1}Q_{\bmu(f)}.1$$
is an isometric ring isomorphism.
\end{thm}

Since $p_n(\phi)$ are linear combinations of $p_n(f)$, we have that
\begin{equation}
 \mathcal{F}_G=\bigotimes_{f\in\Phi}\mathbb C[p_n(f)'s]
 =\bigotimes_{\phi\in\Phi^*}\mathbb C[p_n(\phi)'s]
\end{equation}
where the basis elements $p_{\tilde{\bl}}$ satisfy the inner product \eqref{def:inn}, in particular $\langle p_m(\phi), p_n(\phi')\rangle=m\delta_{m, n}\delta_{\phi, \phi'}$.
Therefore the adjoint operator the differential action $p_n^*(\phi)=n\frac{\partial}{\partial p_n(\phi)}$ of degree $-nd(\phi)$.
As for $\Lambda_{\mathbb C}(f)=Sym[p_n(f)'s]$, we also define $\Lambda_{\mathbb C}(\phi)=Sym[p_n(\phi)'s]$

Following \cite{Jing1, Jing3} we introduce, for each $\phi\in \Phi^*$, the vertex operators $S(\phi, z)$ and $S^*(\phi, z)$ as the following maps:
 $\mathcal{F}_G\to \mathcal{F}_G[z,z^{-1}]$:
\begin{align}\label{e:Schur1}
S(\phi; z)&=\mbox{exp} \left( \sum\limits_{n\geqslant 1} \dfrac{1}{n}p_n(\phi)z^{n} \right) \mbox{exp} \left( -\sum \limits_{n\geqslant 1} \frac{\partial}{\partial p_n(\phi)}z^{-n} \right)=\sum_{n\in\mathbb Z}S_n(\phi)z^{n},\\ \label{e:Schur2}
S^*(\phi; z)&=\mbox{exp} \left(-\sum\limits_{n\geqslant 1} \dfrac{1}{n}p_n(\phi)z^{n} \right) \mbox{exp} \left(\sum \limits_{n\geqslant 1} \frac{\partial}{\partial p_n(\phi)}z^{-n} \right)=\sum_{n\in\mathbb Z}S^*_n(\phi)z^{-n}.
\end{align}
Their components are operators on $\mathcal{F}_G$. We quote the following result from \cite{Jing3}:
\begin{prop} For any partition $\lambda=(\lambda_1, \lambda_2, \dots, \lambda_l)$ and $\phi\in\Phi^*$,
the operator product
$S_{\lambda}(\phi).1= S_{\lambda_1}(\phi)S_{\lambda_2}(\phi)\ldots S_{\lambda_l}(\phi).1$ form an orthonormal basis in
the subspace $\Lambda_{\mathbb C}(\phi)$, and
they are equal to the Schur function $s_{\lambda}(\phi)$ in the power-sum $p_n(\phi)$
\begin{align}\label{e:orth}
\langle S_{\lambda}(\phi).1,S_{\mu}(\phi).1\rangle=\delta_{\lambda, \mu}.
\end{align}
\end{prop}

 More precisely, we also need the following result from \cite{Jing3}.

\begin{lem}\cite{Jing3}\label{L:Jacobi} Let $S(z)$ be the vertex operator of the Schur function. For any partition $\lambda$,
 the coefficient of the Schur function $s_{\lambda}(p_r)$ is the
 $z^{\lambda}$-coefficient of
\begin{align*}
S(z_1)S(z_2)\cdots S(z_l).1=\prod_{1\le i<j\le l}(1-\frac{z_j}{z_i})\sum_{\mu\in \mathcal C_l}h_{\mu}(p_r)z^{\mu}
\end{align*}
where ${\mathcal C}_l$ is the set of compositions of weight $n$.
\end{lem}

For each weighted partition-valued function
$\tilde{\bl}: \Phi^*\longmapsto \mathcal P$ such that
$||\tilde{\bl}||=\sum_{\phi\in \Phi^*}d(\phi)|\bl(\phi)|<\infty$, we denote
\begin{align}
S_{\tilde{\bl}}=\prod_{\phi\in\Phi^*}s_{\tilde{\bl}(\phi)}(\phi)\in
\mathcal{F}_G.
\end{align}\label{inn,S}
Then
\begin{align}
\langle S_{\tilde{\bl}},S_{\tilde{\bmu}}\rangle=\delta_{\tilde{\bl}\tilde{\bmu}}.
\end{align}
In other words, $\{S_{\tilde{\bl}}\}$ form an orthonormal basis for $\mathcal{F}_G$ as well.

\begin{thm} \cite{B} \cite[Chap.IV.5]{Mac}\label{e,char}
    Let $\phi$ be an orbit of $M^*$ and $n\geq 0$, then $e_n(\phi)$ is the characteristic  of a character of $G_{nd(\phi)}$.
\end{thm}
Each Schur function $s_{\tilde{\bl}(\phi)}(\phi)$ is a linear combination of $e_n(\phi)'s$ with integer coefficients,  so $S_{\tilde{\bl}}$ is the  characteristic of a virtual character $\chi^{\tilde{\bl}}$  of $G_n$ by Theorem \ref{e,char}. It follows from  Theorem \ref{thm:ch} that
\begin{align*}
    \langle \chi^{\tilde{\bl}},\chi^{\tilde{\bmu}}\rangle=\delta_{\tilde{\bl}\tilde{\bmu}},
\end{align*}
so $\pm \chi^{\tilde{\bl}}$ must be an irreducible character of $G_n$.
It is known that the number of the irreducible characters of $G_n$ is equal to the conjugacy classes of $G_n$, which is equal to the number of partition value functions $\tilde{\bl}: \Phi^*\to\mathcal{P}$ such that $\|\tilde{\bl}\|=n$. Hence $\{\pm \chi^{\tilde{\bl}}\}$ form  a complete set of irreducible characters of $G_n$, where $||\tilde{\bl}||=n$.

\section{Computation of
characters of $\GL_n(\mathbb F_q)$}

Let's write  $\chi^{\tilde{\bl}}$ as follows (which is yet to be shown a proper character):
\begin{align}\label{e:expansion}
\chi^{\tilde{\bl}}=\sum_{\boldsymbol{\mu}}\chi^{\tilde{\bl}}_{\boldsymbol{\mu}}\pi_{\boldsymbol{\mu}}
\end{align}
summed all over conjugacy classes $\boldsymbol{\mu}$. We claim that
$\chi^{\tilde{\bl}}$ is an irreducible character by showing the coefficient of $\pi_{(1^n)}$ being positive. Applying the characteristic map, \eqref{e:expansion} can be written as
\begin{align}
    S_{\tilde{\bl}}=\sum_{\bmu}\chi_{\bmu}^{\tilde{\bl}}(\prod_{f\in\Phi}q_f^{-n(\bmu(f))}P_{\bmu(f)})
\end{align}
then
\begin{align*}
    \chi_{(1^n)}^{\Tilde{\bl}}=q^{n(1^n)}\langle S_{\tilde{\bl}}, Q_{(1^n)(f_1)}\rangle
\end{align*}
which is denoted as $d(\tilde{\bl})$.

Notice that $q^{n(1^n)}Q_{(1^n)(f_1)}=q^{n(1^n)+|(1^n)|}b_{n}(q^{-1})P_{(1^n)}(f_1)=\psi_n(q)e_n(f_1)$, where $\psi_n(q)=\prod_{i=1}^n(q^i-1)$. Therefore 
\begin{align*}
d(\tilde{\bl})&=\psi_n(q)\langle S_{\tilde{\bl}},e_n(f_1)\rangle\\
&=\psi_n(q)\langle \prod_{\phi\in\Phi^*}S_{\tilde{\bl}(\phi)_1}S_{\tilde{\bl}(\phi)_2}\ldots S_{\tilde{\bl}(\phi)_l}.1, E_n(f_1).1\rangle.
\end{align*}

Note that $\langle \prod_{\phi\in\Phi^*}S_{\tilde{\bl}(\phi)_1}\ldots S_{\tilde{\bl}(\phi)_l}.1, E_n(f_1).1\rangle$ is the coefficient of $z^{\tilde{\bl}(\phi)}w^n$ in the following inner product:
\begin{align*}
&\langle \prod_{\phi\in\Phi^*}S(\phi; z_1)S(\phi; z_2)\ldots S(\phi; z_l).1, E(w;f_1).1\rangle\\
&=\prod_{\phi\in\Phi^*}\prod_{i<j}(1-\frac{z_j}{z_i})\left\langle \exp\left(\sum_{m=1}^{\infty}\frac{p_m(\phi)}m(z_1^m+\cdots +z_l^m)\right),
\exp\left(-\sum_{n=1}^{\infty}\frac{p_n(f_1)}{n}(-w)^n\right)\right\rangle
\end{align*}
where we have used Lemma \ref{L:Jacobi}.
Moving the left exponential operator to the right and using that
$[p_m^*(\phi), p_n(f_1)]=\frac{(-1)^{n-1}n}{q_{\phi}^m-1}\delta_{md(\phi), n}$ (see \eqref{e:comm}), the above is simplified to
\begin{align*}
&\prod_{\phi\in\Phi^*}\prod_{i<j}(1-\frac{z_j}{z_i})\prod_{j=1}^l\exp\left(\sum_{m=1}^{\infty}\frac{1}{m(q_{\phi}^m-1)}(z_jw^{d(\phi)})^m\right)\\
&=\prod_{\phi\in\Phi^*}\prod_{i<j}(1-\frac{z_j}{z_i})\prod_{j=1}^l\prod_{m=1}^\infty\prod_{r=0}^{\infty}\exp\left(\frac{q_{\phi}^{-(r+1)m}}m(z_jw^{d(\phi)})^m\right)\\
&=\prod_{\phi\in\Phi^*}\prod_{i<j}(1-\frac{z_j}{z_i})\sum_{k=0}^{\infty}\sum_{m_1, \ldots, m_l}h_{(m_1, \ldots, m_l)}(q_{\phi}^{-1}, q_{\phi}^{-2}, \ldots)z_1^{m_1}z_2^{m_2}\cdots z_l^{m_l}w^{kd(\phi)}.
\end{align*}
Then it follows from the vertex realization of Schur function
\cite{Jing3} and Lemma \ref{L:Jacobi}
that the coefficient of $z^{\tilde{\bl}}w^n$ is $\prod_{\phi\in\Phi^*}s_{\tilde{\bl}(\phi)}(q_{\phi}^{-1}, q_{\phi}^{-2}, \ldots)$, where $||\tilde{\bl}||=\sum_{\phi}d(\phi)\sum_im_i=n$. By Littlewood's formula
\cite[Ex.2, p.45]{Mac} we have
\begin{align}\label{e:degree2}
d(\tilde{\bl})=\psi_n(q)\prod_{\phi\in \Phi^*}\left(q_{\phi}^{n(\tilde{\bl}(\phi)')}\prod_{x\in\tilde{\bl}(\phi)}(q_{\phi}^{h(x)}-1)^{-1}\right)
\end{align}
where $h(x)$ is the Hook length of the node $x$ inside the Young diagram of shape $\tilde{\bl}(\phi)$.

The formula \eqref{e:degree2} was first obtained by Green
using a different method \cite{jaG} (see also Macdonald's account \cite[pp.284-286]{Mac}).
Since $d(\tilde{\bl})$ is positive, we know that $S_{\tilde{\bl}}$ must correspond to the irreducible character labeled by $\tilde{\bl}$ under the characteristic map.

By the relation \eqref{R:P,Q} and the vertex operator realization, we immediately have
the following result.

\begin{thm}\label{T:irred} For any weighted $\Phi^*$-colored partition
$\tilde{\bl}:\Phi^*\rightarrow \mathcal{P}$ and $\Phi$-colored
partition $\tilde{\bmu}:\Phi\rightarrow {\mathcal P}$ such that
\begin{equation*}
\sum_{\phi\in\Phi^*} d(\phi)|\bl(f)|=\sum_{f\in\Phi} d(f)|\bmu(f)|=n
\end{equation*}
the character value $\chi^{\tilde{\bl}}_{\bmu}$ of the irreducible character $\chi^{\tilde{\bl}}$ of $\GL_n(\mathbb F_q)$ at the conjugacy class $c_{\bmu}$ is given by the matrix coefficient of vertex operators:
\begin{align}\label{e:char}
\chi^{\tilde{\bl}}_{\bmu}&
=\left\langle \prod_{\phi\in\Phi^*}S_{\tilde{\bl}(\phi)}.1, \prod_{f\in\Phi}q_f^{n(\bmu(f))}Q_{\bmu(f)}.1\right\rangle.
\end{align}
\end{thm}

The formula \eqref{e:char} was first obtained in
a different setting by Green \cite{jaG} and accounted in \cite[Ch. IV (6.8)]{Mac}. Some application of lower degree computations was included in \cite{Mor}. Now we describe how to compute the irreducible characters.

By the Frobenius character formula \cite[I.7]{Mac} or \cite[Ch.7.18]{St}, the irreducible character $\overline{\chi}^{\lambda}$
of the symmetric group $\mathfrak S_{|\lambda|}$ is given by:
\begin{equation}\label{e:Frob}
S_{\lambda}(\phi).1=\sum_{\rho}\frac{\overline{\chi}^{\lambda}_{\rho}}{z_{\rho}}p_{\rho}(\phi)
\end{equation}
summed over partitions $\rho\vdash |\lambda|$.

For a fixed $\tilde{\bl}: \Phi^*\longrightarrow \mathcal P$ with $||\bl||=n$, we define the following product for $\tilde{\brho}: \Phi^*\longrightarrow \mathcal P$ such that $||\tilde{\brho}||=n$
\begin{align}\label{e:char-sym}
\overline{\chi}^{\tilde{\bl}}_{\tilde{\brho}}&=\prod_{\phi, \phi'\in\Phi^*}\overline{\chi}^{\tilde{\bl}(\phi)}_{\tilde{\brho}(\phi')}.
\end{align}

Since both $\{q_f^{-|\la|}Q_{{\lambda}}(f).1|\lambda\in\mathcal P\}$ and the power-sums $\{p_{\rho}(f)|\rho\in\mathcal P\}$ are linear bases of the ring $\Lambda_{\mathbb C(q_f)}(f)$
of symmetric functions, we can write for $\lambda, \rho\vdash m$
\begin{equation}\label{e:Green}
q_f^{-|\la|}Q_{{\lambda}}(f).1=\sum_{\rho}\frac{X^{{\lambda}}_{\rho}(q_f^{-1})}{z_{\rho}(q_f^{-1})}p_{\rho}(f),
\end{equation}
where the (modified) coefficients 
$X^{{\lambda}}_{\rho}(q_f^{-1})$ are the so-called {\it Green polynomials} \cite[III.(7.5)]{Mac}.
The equation can also be viewed
as deformation of \eqref{e:Frob} combinatorially.

We define for $\Phi$-colored partitions $\brho, \bl: \Phi \longrightarrow \mathcal P$
\begin{align}\label{e:prod-gr}
\hat{X}^{\bl}_{\brho}&=\prod_{f, f'\in\Phi}q_f^{n(\bl(f))+|\bl(f)|}X^{\bl(f)}_{\brho(f')}(q_{f'}^{-1})\\
Z_{\brho}(q^{-1})&=\prod_{f\in\Phi}z_{\brho(f)}(q_f^{-1})
\end{align}

\begin{thm} For each pair of weighted $\Phi^*$-colored partition $\tilde{\bl}$ and weighted $\Phi$-colored partition $\bmu$ with $||\tilde{\bl}||=||\bmu||=n$, the irreducible
character value $\chi^{\tilde{\bl}}_{\bmu}$ of $\GL(n, \mathbb F_q)$ is given by
\begin{align}\label{e:charfor}
\chi^{\tilde{\bl}}_{\bmu}=
\sum_{\brho}\frac{\overline{\chi}^{\tilde{\bl}}_{{\brho}}\hat{X}^{{\bmu}}_{\brho}}{z_{\brho}} (-1)^{||\brho||-l(\brho)}q^{-||\brho||}\prod_{\phi\in\Phi^*}d(\phi)^{l(\brho(\phi))}\prod_{i=1}^{l(\brho(f(\phi)))}(\phi, f(\phi))_{d(\phi)\brho_i(\phi)}
\end{align}
summed over all $\Phi^*$-colored weighted partitions $\brho$ such that $||\brho||=\sum_{\phi\in\Phi^*}d(\phi)|\brho(\phi)|=n$,
and $\overline{\chi}^{\tilde{\bl}}_{\brho}$ are the products of the irreducible characters of the Young subgroups $\prod_{f\in\Phi}{\mathfrak S}_{\brho(f)}$ of the symmetric group $\mathfrak S_{\sum_{\phi\in\Phi^*}|\brho(f)|}$ and
$\hat{X}^{{\bmu}}_{\brho}$ are products of the Green polynomials, and $f(\phi)\in\Phi$ is the corresponding orbit of $\phi\in\Phi^*$.
\end{thm}
\begin{proof} By Theorem \ref{T:irred} it follows that the irreducible characters are given by
\begin{align*}
\chi^{\tilde{\bl}}_{\bmu}
&=\left\langle \prod_{\phi\in\Phi^*}S_{\tilde{\bl}(\phi)}.1, \prod_{f\in\Phi}q_f^{n(\bmu(f))}Q_{\bmu(f)}.1\right\rangle\\
&=\left\langle \sum_{\tilde{\brho}}\overline{\chi}^{\tilde{\bl}}_{\tilde{\brho}}\frac{p_{\tilde{\brho}}}{z_{\tilde{\brho}}},
\sum_{\brho'}\frac{\hat{X}^{{\bmu}}_{\brho'}}{Z_{\brho'}(q^{-1})}p_{\brho'}          \right\rangle\\
&=\sum_{\tilde{\brho}, \brho'}\frac{\overline{\chi}^{\tilde{\bl}}_{\tilde{\brho}}\hat{X}^{{\bmu}}_{\brho'}}{z_{\tilde{\brho}}Z_{\brho'}(q^{-1})} \langle p_{\tilde{\brho}}, p_{\brho'}\rangle
\end{align*}
One can pick a primitive $(q^{n!}-1)$th root of unity $\omega$, so $\omega_m=\omega^{\frac{q^{n!}-1}{q^m-1}}$ is the primitive $(q^m-1)$th root of unity and
$M_m=\langle \omega_m\rangle$, therefore the irreducible character $\chi_m$ given by $\chi_m(\omega_m)=\omega_m$ will generate all irreducible linear characters of $M_m^*$, i.e. $M_m^*=\langle\chi_m\rangle\simeq \langle \omega_m\rangle=M_m$ as cyclic groups.
This implies that
there exists a one-one correspondence between $F$-orbits of $M_n$ and $M_n^*$, subsequently there exists a one-to-one correspondence
between the weighted $\Phi^*$-colored partitions of $n$ and the weighted $\Phi$-colored partitions of $n$
such that $\tilde{\bl}$ and $\bmu$ match in the sense that $d(\phi)\bl(\phi)=d(f)\bmu(f)$. Moreover, for fixed degree $d(\phi)=d(f)$, the correspondence is also
one-to-one. It follows from \eqref{e:commutator2} that for two
partitions $\lambda=(\lambda_1, \ldots, \lambda_l)$ and
$\mu=(\mu_1, \ldots, \mu_k)$
\begin{align*}
&\langle p_{\lambda_1}(\phi)\cdots p_{\lambda_l}(\phi), p_{\mu_1}(f)\cdots p_{\mu_k}(f)\rangle\\
&=(-1)^{d(\phi)\lambda_1-1}\frac{\lambda_1d(\phi)(\phi, f)_{\lambda_1d(\phi)}}
{q^{\lambda_1}_{\phi}-1}\langle p_{\lambda_2}(\phi)\cdots p_{\lambda_l}(\phi), p_{\mu_2}(f)\cdots p_{\mu_k}(f)\rangle \delta_{\lambda_1d(\phi), \mu_1d(f)}\\
&=\cdots =\delta_{d(\phi)\lambda, d(f)\mu}(-1)^{d(\phi)|\lambda|-l(\lambda)}z_{\lambda}(q_{\phi}^{-1})q^{-d(\phi)|\lambda|}d(\phi)^l\prod_{i=1}^l(\phi, f)_{\lambda_id(\phi)}.
\end{align*}
Taking account of degree matching of $F$-orbits, it concludes that
$$\langle p_{\tilde{\brho}}, p_{\brho'}\rangle=\delta_{\tilde{\brho}, \brho'}(-1)^{||\brho'||-l(\brho')}Z_{\brho'}(q^{-1})q^{-||\brho'||}\prod_{\phi\in\Phi^*}d(\phi)^{l(\brho'(\phi))}\prod_{i=1}^{l(\brho'(\phi))}(\phi, f)_{d(\phi)\brho'_i(\phi)},$$
then the formula \eqref{e:charfor} of the irreducible character $\chi^{\tilde{\bl}}$ follows (replacing $\brho'$ by $\brho$).
\end{proof}

\begin{rem} The modified Green polynomials $X^{\lambda}_{\mu}(t)$ for $t=q^{-1}$ can be effectively computed by the method of vertex operators \cite{JL}, where the reader can find the Murnaghan-Makayama formula and various compact formulas.
We also note that the irreducible character $\overline{\chi}^{\lambda}_{\mu}=X^{\lambda}_{\mu}(0)$, which is also conveniently obtained by the vertex operator calculus \cite{JL} with $t=0$.
\end{rem}

\bigskip

\bigskip
\bibliographystyle{plain}

\end{document}